\documentclass[3p]{elsarticle}

\usepackage{amsmath,amsthm,amssymb,mathrsfs}
\usepackage{enumerate}
\usepackage{slashed}
\usepackage{txfonts,dsfont}
\usepackage{aliascnt}
\usepackage[breaklinks,colorlinks]{hyperref}

\makeatletter
\newcommand{\autorefcheckize}[1]{%
  \expandafter\let\csname @@\string#1\endcsname#1%
  \expandafter\DeclareRobustCommand\csname relax\string#1\endcsname[1]{%
    \csname @@\string#1\endcsname{##1}\wrtusdrf{##1}}%
  \expandafter\let\expandafter#1\csname relax\string#1\endcsname
}
\makeatother

\theoremstyle{plain}

\newtheorem{theorem}{Theorem}[section]

\newaliascnt{lem}{theorem}
\newtheorem{lem}[lem]{Lemma}
 \aliascntresetthe{lem}

\newaliascnt{cor}{theorem}

 \aliascntresetthe{cor}

\newaliascnt{prop}{theorem}

 \aliascntresetthe{prop}

\theoremstyle{remark}

\newtheorem{rem}{Remark}[section]

\newtheorem*{claim}{Claim}

\theoremstyle{definition}

\newtheorem{eg}{Example}[section]

\numberwithin{equation}{section}

\newcommand{\abs}[1]{\left\lvert#1\right\rvert}
\newcommand{\set}[1]{\left\{#1\right\}}
\newcommand{\hin}[2]{\left\langle#1,#2\right\rangle}

\newcommand*{\To}{\longrightarrow}
\newcommand*{\Rmn}[1]{\uppercase\expandafter{\romannueral#1}}
\newcommand*{\dif}{\mathop{}\!\mathrm{d}}

\DeclareMathOperator{\trace}{tr}

\journal{}

\allowdisplaybreaks

\begin{document}

\begin{frontmatter}

\title{A new characterization of the Calabi torus  in the unit sphere \tnoteref{lsy}}

\author[whu1,whu2]{Yong Luo}
\ead{yongluo@whu.edu.cn}

\author[whu1,whu2]{Linlin Sun\corref{sll1}}
\ead{sunll@whu.edu.cn}

\author[xmu]{Jiabin Yin}
\ead{jiabinyin@126.com}

\tnotetext[lsy]{This work was partially supported by the NSF of China (Nos. 11501421, 11801420, 11771404) and the Youth Talent Training Program of Wuhan University. The second author thanks the Max Planck Institute for Mathematics in the Sciences for good working conditions when this work carried out.}

\address[whu1]{School of Mathematics and Statistics, Wuhan University, Wuhan 430072, China}
\address[whu2]{Hubei Key Laboratory of Computational Science, Wuhan University, Wuhan, 430072, China}

\address[xmu]{School of Mathematical Sciences, Xiamen University, Xiamen, 361005, China}

\cortext[sll1]{Corresponding author.}

\begin{abstract}
In this paper, we study the rigidity theorem of closed minimally immersed Legendrian submanifolds in the unit sphere. Utilizing the maximum principle, we obtain a new characterization of the Calabi torus in the unit sphere which is the minimal  Calabi product Legendrian immersion of a point and the totally geodesic Legendrian sphere. We also establish an optimal Simons' type integral inequality in terms of the second fundamental form of  three dimensional  closed minimal Legendrian submanifolds in the unit sphere. 

\end{abstract}


\end{frontmatter}


\section{Introduction}
Let $M$ be an $n$-dimensional closed minimally immersed submanifold in the unit sphere $\mathbb{S}^{n+m}$ of dimension $n+m$. Let $\mathbf{B}$ be the second fundamental form of this immersion.  
Simons \cite{Si}, Chern, Do Carmo and Kobayashi \cite{CCK}, Lawson \cite{La} proved that under the pinching condition   $\abs{\mathbf{B}}^2\leq\frac{n}{2-\frac{1}{m}}$,  $M$  must be either one of the  Clifford minimal tori $\mathbf{S}^{p}\left(\sqrt{\frac{p}{n}}\right)\times\mathbf{S}^{n-p}\left(\sqrt{\frac{n-p}{n}}\right)$ in $\mathbb{S}^{n+1}$ or the Veronese surface in $\mathbb{S}^4$ unless $M$ is the totally geodesic sphere $\mathbb{S}^n$ in $\mathbb{S}^{n+1}$. Li and Li \cite{LL} improved Simons' pinching constant to $\frac{2n}{3}$ for higher codimension $m\geq3$. Chen and Xu \cite{CX} obtained the same pinching constant independently  by using a different method. These rigidity results mentioned above can be viewed as an intrinsic rigidity theorem for pinching of scalar curvature according to the Gauss equation. The intrinsic  rigidity theorem for  pinching of sectional curvature was obtained by Yau \cite{Yau}, for pinching of Ricci curvature was obtained by Ejiri \cite{Ejiri79}.
The extrinsic  rigidity theorem for pinching of second fundamental form was obtained by Gauchman \cite{Gau1}.

There are many papers on the particularly interesting case of closed minimal Legendrian submanifolds in the unit sphere $\mathbb{S}^{2n+1}$ or closed minimal Lagrangian submanifolds in $\mathbb{CP}^n$ (for an incomplete list, see e.g. \cite{CO,YMI,NT82,BO1} for pinching of scalar curvature, \cite{Yau,YKM,BO2,Ogi,Ur1,DV} for pinching of sectional curvature, \cite{MRU} for pinching of Ricci curvature). Inspired by papers of Ros \cite{Ros, Ros85} on pinching and rigidity of K\"ahler submanifolds, Gauchman \cite{Gau2} and Xia \cite{Xia1} studied pinching of the geometric quantity
\begin{align*}
    \Theta(p)\coloneqq\max_{X\in T_pM,\  \abs{X}=1}\abs{\mathbf{B}\left(X,X\right)}
\end{align*}
for closed Lagrangian submanifolds in $\mathbb{CP}^n$. In particular, Xia \cite{Xia1} proved that if $\Theta^2\leq\frac{1}{2}$, then either $M$ is totally geodesic or $\Theta^2\equiv\frac{1}{2}$ and the last case was classified completely. 

These curvature pinching and characterization results were proved based on analysis of Simons' type  formula. This formula is related to a special sort of submanifolds, those that have parallel second fundamental form. Lagrangian submanifolds in $\mathbb{CP}^n$ with parallel second fundamental form were completely classified by Naitoh \cite{Nai,Nai83} for irreducible case and Dillen, Li, Vrancken  and Wang \cite{DLVW} for general case. The Classification theorem of Dillen, Li, Vrancken  and Wang claims that Lagrangian submanifolds with parallel second fundamental form in $\mathbb{CP}^n$ are one of the following: 
\begin{enumerate}[a)]
    \item totally geodesic submanifolds;
    \item  embedded submanifolds which are locally congruent to one of the following
standard embeddings in $\mathbb{CP}^n$:
\begin{align*}
    \mathrm{SU}(k)/\mathrm{SO}(k),\quad&\quad n=(k-1)(k+2)/2,\quad k\geq3,\\ \mathrm{SU}(k),\quad&\quad n=k^2-1,\quad k\geq3,\\
    \mathrm{SU}(2k)/\mathrm{Sp}(k) ,\quad&\quad n=2k^2-k-1,\quad k\geq3,\\ \mathrm{E}_6/\mathrm{F}_4,\quad&\quad n=26;
\end{align*}
    \item locally a finite Riemannian covering of the unique flat torus minimally embedded in $\mathbb{CP}^2$ with parallel second fundamental form;
    \item locally the Calabi product of a point with a lower dimensional Lagrangian submanifold with parallel second fundamental form;
\item locally the Calabi product of two lower dimensional Lagrangian submanifolds  with parallel second fundamental form.
\end{enumerate}
The  examples of a)-c) are minimal Lagrangian submanifolds, but examples of d)-e) contain both minimal and non-minimal ones. Furthermore the unique minimal submanifold in d) is the so called Calabi torus, which is the image of \autoref{exm} by the Hopf fibration of $\mathbb{S}^{2n+1}$ to $\mathbb{CP}^n$. 

The above mentioned papers in paragraph 2 gave various curvature pinching and characterization results for compact minimal Lagrangian submanifolds of a) (cf. \cite{CO,YMI,BO1,BO2,Ogi,Ur1}); a) and c) (cf. \cite{,NT82,Yau,YKM, DV}); and a), b) when $k=3$ and c) (cf. \cite{Gau2,Xia1}). Nevertheless according to our knowledge such kind of result was missed for the examples of d) and e). Bewaring of this, Luo and Sun \cite{LS18}  conjectured that if $M$ is a closed minimal Legendrian submanifold in $\mathbb{S}^{2n+1}$ and $\abs{\mathbf{B}}^2\leq\frac{(n+2)(n-1)}{n}$, then $M$ is either the totally geodesic sphere or the Calabi torus (cf. \autoref{exm}). In this paper we aim to get a curvature pinching and characterization result for the Calabi torus and we obtain the following theorem.
\begin{theorem}\label{thm1}
Let $M$ be a closed minimal Legendrian submanifold  in the unit sphere $\mathbb{S}^{2n+1}(n\geq2)$ and $\mathbf{B}$ be its second fundamental form. Assume the following pinching condition holds pointwisely,
\begin{align}\label{con}
\abs{\mathbf{B}}^2\leq\frac{n+2}{\sqrt{n}}\Theta,
\end{align}
then $M$ is either the totally geodesic sphere or the Calabi torus.

If $n=3$, the pinching condition can be changed weakly to
\begin{align*}
    \abs{\mathbf{B}}^2\leq 2+\Theta^2.
\end{align*}
\end{theorem}
\begin{rem}
Checking the proof step by step, one obtains that a closed minimal Legendrian submanifold $M$ in the unit sphere $\mathbb{S}^7$ with $\Theta^2\leq\frac{2}{3}$ must be the totally geodesic sphere. Therefore we improve Xia's result \cite{Xia1} for minimal  Lagrangian submanifolds in $\mathbb{CP}^3$.  Consequently, in case of dimension 3 we obtain Li and Li's type pinching condition \cite{LL} , i.e., $M$ is totally geodesic if $\abs{\mathbf{B}}^2\leq\frac{8}{3}$. 
\end{rem}

It is worthy to be mentioned that this theorem could be  parallel stated for closed minimal Lagrangian submanifolds in $\mathbb{CP}^n$, due to the well known correspondence of minimal Legendrian submanifolds in $\mathbb{S}^{2n+1}$ and minimal Lagrangian submanifolds in $\mathbb{CP}^n$ (cf. \cite{CLU}), or by proofs with similar argument.

We prove  \autoref{thm1} by applying a maximum principle for tensor and a Simons' type formula of closed minimal Legendrian submanifolds in the unit sphere. We will also use an integral method to get an integral inequality of three dimensional closed minimal Legendrian submanifolds in $\mathbb{S}^7$, which implies another pinching and rigidity result for them (cf.  \autoref{thm2}).

In section 2 we give some preliminaries on Legendrian submanifolds of the unit sphere, including some details on the Calabi torus and a Simons' type formula.  In section 3 we prove an integral inequality of closed Legendrian submanifolds in $\mathbb{S}^7$. \autoref{thm1} is proved in section 4. In the Appendix we prove an integral inequality for Lagrangian submanifolds in the nearly K\"ahler $\mathbb{S}^6$ by  a similar argument used in the proof of \autoref{thm2}, which improves the main theorem of Hu, Yin and Yin \cite{HYY}.

\section{Preliminaries}
Here we briefly record several facts about Legendrian submanifolds in the unit sphere. We refer the reader to consult \cite{Bl} for more materials about the contact geometry.

Let $M$ be a closed $n$-dimensional submanifold of the unit sphere $\mathbb{S}^{2n+1}\subset\mathbb{C}^{n+1}$. We say that $M$ is Legendrian if 
\begin{align*}
    JTM\subset T^{\bot}M,\quad JF\in\Gamma\left(T^{\bot}M\right) 
\end{align*}
where $F: M\To\mathbb{S}^{2n+1}$ is the position vector and $J$ is the complex structure of $\mathbb{C}^{n+1}$. We say that $M$ is a minimal Legenbdrian subamnifold of $\mathbb{S}^{2n+1}$ if $M$ is a minimal and Legendrian submanifold of $\mathbb{S}^{2n+1}$. Define
\begin{align*}
    \sigma\left(X,Y,Z\right)\coloneqq\hin{\mathbf{B}\left(X,Y\right)}{JZ},\quad\forall X,Y,Z\in TM.
\end{align*}
The Weingarten equation implies that
\begin{align*}
    \sigma\left(X,Y,Z\right)=\sigma\left(Y,X,Z\right).
\end{align*}
Moreover, by definition, one can check that $\sigma$ is a three order symmetric tensor, i.e.,
\begin{align*}
    \sigma\left(X,Y,Z\right)=\sigma\left(Y,X,Z\right)=\sigma\left(X,Z,Y\right).
\end{align*}
The Gauss equation, Codazzi equation and Ricci equation becomes
\begin{align*}
    R\left(X,Y,Z,W\right)=&\hin{X}{Z}\hin{Y}{W}-\hin{X}{W}\hin{Y}{Z}\\
    &+\sigma\left(X,Z,e_i\right)\sigma\left(Y,W,e_i\right)-\sigma\left(X,W,e_i\right)\sigma\left(Y,Z,e_i\right),\\
    \left(\nabla_{X}\sigma\right)\left(Y,Z,W\right)=&\left(\nabla_{Y}\sigma\right)\left(X,Z,W\right),\\
    R^{\bot}\left(X,Y,JZ,JW\right)=&R\left(X,Y,Z,W\right),
\end{align*}
where $\set{e_i}$ is an orthonormal basis of $TM$. The Codazzi equation implies
\begin{align*}
    \left(\nabla_{X}\sigma\right)\left(Y,Z,W\right)=\left(\nabla_{Y}\sigma\right)\left(X,Z,W\right)=\left(\nabla_{X}\sigma\right)\left(Z,Y,W\right)=\left(\nabla_{X}\sigma\right)\left(Y,W,Z\right),
\end{align*}
i.e., $\nabla\sigma$ is a fourth order symmetric tensor.

A typical example of closed minimal Legendrian submanifold in the unit sphere is the totally real and totally geodesic sphere $\mathbb{S}^{n}\subset\mathbb{S}^{2n+1}$. One can use the Calabi product (cf. \cite{CMU}) to obtain new examples of minimal Legendrian submanifolds.  

\begin{eg}[Calabi product Legendrian immersions]\label{exm}
Let 
\begin{align*}
    \gamma=(\gamma_1,\gamma_2):\mathbb{S}^1\To\mathbb{S}^{3},\quad t\mapsto\left(\sqrt{\frac{n}{n+1}}\exp\left(\sqrt{-1}\sqrt{\frac{1}{n}}t\right),\sqrt{\frac{1}{n+1}}\exp\left(-\sqrt{-1}\sqrt{n}t\right)\right)
\end{align*} be the totally geodesic sphere. Then $ F\coloneqq\left(\gamma_1\phi,\gamma_2\right):\mathbb{S}^1\times \mathbb{S}^{n-1}\To\mathbb{S}^{2n+1}$ is a minimal Legendrian immersion. Denote by $M\coloneqq F\left(\mathbb{S}^1\times \mathbb{S}^{n-1}\right)$. One can choose a local orthonormal frames of $TM$ such that the second fundamental form $\mathbf{B}$ satisfies 
\begin{align*}
    \mathbf{B}\left(e_1,e_j\right)=&-\sqrt{\dfrac{1}{n}}Je_j+\delta_{1j}\sqrt{n}Je_1,\\
    \mathbf{B}\left(e_i,e_j\right)=&-\delta_{ij}\sqrt{\dfrac{1}{n}}Je_1,\quad i, j\in\set{2,\dotsc,n}.
\end{align*}
We call this minimal Legendrian submanifold $M$ the  \emph{Calabi torus}. One can check that
\begin{align*}
    \abs{\mathbf{B}}^2=\dfrac{(n-1)(n+2)}{n},\quad\max_{X\in TM,\  \abs{X}=1}\abs{\mathbf{B}\left(X,X\right)}=\dfrac{n-1}{\sqrt{n}}.
\end{align*}

\end{eg}

We will need the following Simons' identity (cf. \cite{Si}, see also \cite{CO, YMI}).
\begin{lem}[Simons' identity]
Assume that $M$ is a minimal Legendrian submanifold in $\mathbb{S}^{2n+1}$. Then
\begin{equation}\label{eq:simons}
\begin{split}
\Delta\sigma_{ijk}\coloneqq&\sum_l\sigma_{ijk,ll}\\
=&(n+1)\sigma_{ijk}+2\sum_{l,s,t}\sigma_{isl}\sigma_{jlt}\sigma_{kts}-\sum_{l,s,t}\sigma_{tli}\sigma_{tls}\sigma_{jks}
-\sum_{l,s,t}\sigma_{tlj}\sigma_{tls}\sigma_{iks}-\sum_{l,s,t}\sigma_{tlk}\sigma_{tls}\sigma_{ijs}.
\end{split}
\end{equation}
Consequently,
\begin{align}\label{eq:simons1}
\dfrac12\Delta\abs{\sigma}^2=&\abs{\nabla\sigma}^2+\left(n+1\right)\abs{\sigma}^2-\sum_{i,j}\hin{\sigma_i}{\sigma_j}^2-\sum_{i,j}\abs{[\sigma_i,\sigma_j]}^2,
\end{align}
where $\sigma_i=\left(\sigma_{ijk}\right)_{1\leq j,k\leq n}$.
\end{lem}
\begin{proof}
The Ricci identity yields
\begin{align*}
\sigma_{ijk,lm}=\sigma_{ijk,ml}+\sum_{t}\sigma_{tjk}R_{tilm}+\sum_{t}\sigma_{itk}R_{tjlm}+\sum_{t}\sigma_{ijt}R_{tklm}.
\end{align*}
Therefore,
\begin{align*}
\Delta\sigma_{ijk}=&\sum_{l}\sigma_{ijk,ll}\\
=&\sum_{l}\sigma_{ijl,kl}\\
=&\sum_{l}\sigma_{ijl,lk}+\sum_{l,t}\sigma_{tjl}R_{tikl}+\sum_{l,t}\sigma_{itl}R_{tjkl}+\sum_{l,t}\sigma_{ijt}R_{tlkl}\\
=&\mu_{i,jk}+\sum_{l,t}\sigma_{tjl}R_{tikl}+\sum_{l,t}\sigma_{itl}R_{tjkl}+\sum_{l,t}\sigma_{ijt}R_{tlkl}.
\end{align*}
Here  $\mu_i=\trace\sigma_i$.
Thus by the Gauss equation,
\begin{align*}
\Delta\sigma_{ijk}=&\mu_{i,jk}+\sum_{l,t}\sigma_{tjl}\left(\delta_{tk}\delta_{il}-\delta_{tl}\delta_{ik}+\sigma_{tks}\sigma_{ils}-\sigma_{tls}\sigma_{iks}\right)\\
&+\sum_{l,t}\sigma_{til}\left(\delta_{tk}\delta_{jl}-\delta_{tl}\delta_{jk}+\sigma_{tks}\sigma_{jls}-\sigma_{tls}\sigma_{jks}\right)\\
&+\sum_{l,t}\sigma_{ijt}\left((n-1)\delta_{tk}+\sigma_{tks}\sigma_{lls}-\sigma_{tls}\sigma_{lks}\right)\\
=&\mu_{i,jk}+\sigma_{ijk}-\mu_{j}\delta_{ik}+\sum_{l,s,t}\sigma_{tjl}\left(\sigma_{tks}\sigma_{ils}-\sigma_{tls}\sigma_{iks}\right)\\
&+\sigma_{ijk}-\mu_{i}\delta_{jk}+\sum_{l,s,t}\sigma_{til}\left(\sigma_{tks}\sigma_{jls}-\sigma_{tls}\sigma_{jks}\right)\\
&+(n-1)\sigma_{ijk}+\sum_{l,s,t}\sigma_{ijt}\left(\sigma_{tks}\mu_{s}-\sigma_{tls}\sigma_{lks}\right)\\
=&\mu_{i,jk}-\mu_{i}\delta_{jk}-\mu_{j}\delta_{ik}+\sum_{s,t}\sigma_{ijt}\sigma_{tks}\mu_{s}\\
&+(n+1)\sigma_{ijk}+2\sum_{l,s,t}\sigma_{tjl}\sigma_{tks}\sigma_{ils}-\sum_{l,s,t}\sigma_{tjl}\sigma_{tls}\sigma_{iks}
-\sum_{l,s,t}\sigma_{til}\sigma_{tls}\sigma_{jks}-\sum_{l,s,t}\sigma_{tls}\sigma_{lks}\sigma_{ijt}
\\&=(n+1)\sigma_{ijk}+2\sum_{l,s,t}\sigma_{tjl}\sigma_{tks}\sigma_{ils}-\sum_{l,s,t}\sigma_{tjl}\sigma_{tls}\sigma_{iks}
-\sum_{l,s,t}\sigma_{til}\sigma_{tls}\sigma_{jks}-\sum_{l,s,t}\sigma_{tls}\sigma_{lks}\sigma_{ijt},
\end{align*}
where we used the fact that $\mu_i=0$ since $M$ is minimal.
\end{proof}

\section{An integral inequality for the three dimensional case}
In this section we prove an integral inequality for closed three dimensional minimal Legendrian submanifolds in $\mathbb{S}^7$, which is inspired by a recent paper of Hu,Yin and Yin \cite{HYY}.

\begin{theorem}\label{thm2}
Let $M$ be a closed minimal Legendrian submanifold  in the unit sphere $\mathbb{S}^{7}$, then
\begin{align}\label{intI}
\int_{M}\abs{\mathbf{B}}^2\left(\abs{\mathbf{B}}^2-\dfrac{10}{7}\left(1+\Theta^2\right)\right)\geq0.
\end{align}
Consequently, if
\begin{align*}
  \abs{\mathbf{B}}^2\leq\dfrac{10}{7}\left(1+\Theta^2\right),
\end{align*}
then $M$ is either the totally geodesic sphere or the  Calabi torus.
\end{theorem}

\begin{rem}
We would like to point out that though the pinching result we obtain by integral estimates  here is actually weaker than that we obtain in \autoref{thm1}, we can not get any integral inequality like \eqref{intI}  by the maximum principle used in the proof of \autoref{thm1}. Furthermore here we slightly refine the argument in \cite{HYY} and use it to give an improvement of the main theorem in \cite{HYY}, please see \autoref{thm3} in the Appendix for details. It seems that the maximum principle is not applicable in proving pinching result for Lagrangian submanifolds in the nearly K\"ahler $\mathbb{S}^6$.

\end{rem}

\begin{proof}[Proof of \autoref{thm2}]
Consider an algebraic curvature $\hat R$ defined by
\begin{align*}
\hat R_{ijkl}=\hin{\mathbf{B}\left(e_i,e_k\right)}{\mathbf{B}\left(e_j,e_l\right)}-\hin{\mathbf{B}\left(e_i,e_l\right)}{\mathbf{B}\left(e_j,e_k\right)},
\end{align*}
i.e.,
\begin{align*}
\hat R_{ijkl}=\sum_{a}\left(\sigma_{ika}\sigma_{jla}-\sigma_{ila}\sigma_{jka}\right)=[\sigma_i,\sigma_j]_{kl}.
\end{align*}
The algebraic Ricci curvature $\hat Ric$ and the algebraic scalar curvature $\hat S$ are given by
\begin{align*}
\hat Ric_{ij}=\sum_{a}\hat R_{iaja}=-\hin{\sigma_i}{\sigma_j},\quad\hat S=\sum_{i}\hat Ric_{ii}=-\abs{\sigma}^2.
\end{align*}
We therefore can rewrite the Simons' identity \eqref{eq:simons1} as follows
\begin{align}\label{eq:simons-3}
\dfrac12\Delta\abs{\sigma}^2=&\abs{\nabla\sigma}^2+\left(n+1\right)\abs{\sigma}^2-\abs{\hat Ric}^2-\abs{\hat R}^2.
\end{align}

Recall the orthogonal decomposition for the algebraic curvature
\begin{align*}
\hat R=\hat W+\dfrac{1}{n-2}\mathring{\hat R}ic\circledwedge g+\dfrac{\hat S}{2n(n-1)}g\circledwedge g,
\end{align*}
where $\hat W$ is the algebraic Weyl curvature and $\mathring{\hat R}ic=\hat Ric-\frac{\hat S}{n}g$ is the traceless algebraic Ricci curvature. We have the following identity
\begin{align*}
\abs{\hat R}^2=&\abs{\hat W}^2+\dfrac{4\abs{\mathring{\hat R}ic}^2}{n-2}+\dfrac{2\hat S^2}{n(n-1)}\\
=&\abs{\hat W}^2+\dfrac{4\abs{\hat Ric}^2}{n-2}-\dfrac{2\hat S^2}{(n-1)(n-2)}.
\end{align*}
For $n=3$, the algebraic Weyl curvature $\hat W$ vanishes. It follows from \eqref{eq:simons-3} that
\begin{align*}
\dfrac12\Delta\abs{\sigma}^2=&\abs{\nabla\sigma}^2+4\abs{\sigma}^2-5\abs{\hat Ric}^2+\abs{\hat S}^2\\
=&\abs{\nabla\sigma}^2+4\abs{\sigma}^2-5\sum_{i,j=1}^3\hin{\sigma_i}{\sigma_j}^2+\abs{\sigma}^4.
\end{align*}

At a point $p$, choose $e_1$ such that
\begin{align*}
    \sigma_{111}=\max_{X\in S_pM^3}\sigma\left(X,X,X\right),
\end{align*}
then $\sigma_{112}=\sigma_{113}=0$. Then we choose $\set{e_2, e_3}$ such that $\sigma_{123}=0$. In other words, we may assume
\begin{align*}
\sigma_1=\begin{pmatrix}\lambda_1+\lambda_2&0&0\\
0&-\lambda_1&0\\
0&0&-\lambda_2
\end{pmatrix},\quad \sigma_2=\begin{pmatrix}0&-\lambda_1&0\\
-\lambda_1&\mu_1&\mu_2\\
0&\mu_2&-\mu_1
\end{pmatrix},\quad \sigma_3=\begin{pmatrix}0&0&-\lambda_2\\
0&\mu_2&-\mu_1\\
-\lambda_2&-\mu_1&-\mu_2
\end{pmatrix}.
\end{align*}

A direct calculation yields
\begin{align*}
\abs{\sigma}^2=&4\lambda_1^2+4\lambda_2^2+2\lambda_1\lambda_2+4\left(\mu_1^2+\mu_2^2\right)\\
=&\dfrac{5}{2}\left(\lambda_{1}+\lambda_{2}\right)^2+\dfrac{3}{2}\left(\lambda_{1}-\lambda_{2}\right)^2+4\left(\mu_1^2+\mu_2^2\right),\\
\sum_{i,j}\hin{\sigma_i}{\sigma_j}^2=&4\left(\lambda_1^2+\lambda_2^2+\lambda_1\lambda_2\right)^2+4\left(\lambda_1^2+\mu_1^2+\mu_2^2\right)^2+4\left(\lambda_2^2+\mu_1^2+\mu_2^2\right)^2\\
&+2\left(\lambda_1-\lambda_2\right)^2\left(\mu_1^2+\mu_2^2\right)\\
=&\dfrac{11}{4}\left(\lambda_1+\lambda_2\right)^4+\dfrac{3}{4}\left(\lambda_1-\lambda_2\right)^4+8\left(\mu_1^2+\mu_2^2\right)^2\\
&+\dfrac{9}{2}\left(\lambda_1+\lambda_2\right)^2\left(\lambda_1-\lambda_2\right)^2+4\left(\lambda_1+\lambda_2\right)^2\left(\mu_1^2+\mu_2^2\right)+6\left(\lambda_1-\lambda_2\right)^2\left(\mu_1^2+\mu_2^2\right).
\end{align*}

Set
\begin{align*}
x=\left(\lambda_1+\lambda_2\right)^2,\quad y=\left(\lambda_1-\lambda_2\right)^2,\quad z=4\left(\mu_1^2+\mu_2^2\right),
\end{align*}
then
\begin{align*}
\abs{\sigma}^2=&\dfrac{5}{2}x+\dfrac{3}{2}y+z,\\
\sum_{i,j}\hin{\sigma_i}{\sigma_j}^2=&\dfrac{11}{4}x^2+\dfrac{3}{4}y^2+\dfrac12z^2+\dfrac{9}{2}xy+xz+\dfrac32yz,\\
\dfrac{1}{5}\abs{\sigma}^4=&\dfrac{5}{4}x^2+\dfrac{9}{20}y^2+\dfrac{1}{5}z^2+\dfrac{3}{2}xy+xz+\dfrac{3}{5}yz.
\end{align*}
For every $\kappa$, we have
\begin{align*}
\sum_{i,j}\hin{\sigma_i}{\sigma_j}^2-\dfrac{1}{5}\abs{\sigma}^4=&\dfrac{3}{2}x^2+\dfrac{3}{10}y^2+\dfrac{3}{10}z^2+3xy+\dfrac{9}{10}yz\\
=&\dfrac32\left(x+\kappa y+\dfrac{2\kappa}{3}z\right)\left(x+\dfrac{3}{5}y+\dfrac{2}{3}z\right)+\dfrac{3}{10}y^2+\dfrac{3}{10}z^2+3xy+\dfrac{9}{10}yz\\
&-\dfrac{9\kappa}{10}y^2-\dfrac{2\kappa}{3}z^2-\dfrac{15\kappa+9}{10}xy-\left(1+\kappa\right)xz-\dfrac{8\kappa}{5}yz\\
=&\dfrac32\left(x+\kappa y+\dfrac{2\kappa}{3}z\right)\left(x+\dfrac{3}{5}y+\dfrac{2}{3}z\right)\\
&-\dfrac{9\kappa-3}{10}y^2-\dfrac{20\kappa-9}{30}z^2-\dfrac{15\kappa-21}{10}xy-\left(1+\kappa\right)xz-\dfrac{16\kappa-9}{10}yz.
\end{align*}
For $\kappa\geq\frac{7}{5}$,
\begin{align*}
\sum_{i,j}\hin{\sigma_i}{\sigma_j}^2-\dfrac{1}{5}\abs{\sigma}^4\leq&\dfrac32\left(x+\kappa y+\dfrac{2\kappa}{3}z\right)\left(x+\dfrac{3}{5}y+\dfrac{2}{3}z\right)\\
=&\dfrac{2\kappa}{5}\left(\abs{\sigma}^2-\dfrac{5\kappa-3}{2\kappa}\Theta^2\right)\abs{\sigma}^2.
\end{align*}
Therefore we have the estimate
\begin{align*}
\dfrac12\Delta\abs{\sigma}^2\geq&\abs{\nabla\sigma}^2+2\kappa\left(\dfrac{2}{\kappa}+\dfrac{5\kappa-3}{2\kappa}\Theta^2-\abs{\sigma}^2\right)\abs{\sigma}^2,\quad\forall \kappa\geq\dfrac{7}{5}.
\end{align*}
In particular, take $\kappa=\frac{7}{5}$ to obtain
\begin{align*}
    \dfrac12\Delta\abs{\sigma}^2\geq&\abs{\nabla\sigma}^2+\frac{14}{5}\left(\dfrac{10}{7}\left(1+\Theta^2\right)-\abs{\sigma}^2\right)\abs{\sigma}^2.
\end{align*}
Integration by parts, we prove the first claim of the theorem. 

If $\abs{\mathbf{B}}^2\leq\frac{10}{7}\left(1+\Theta^2\right)$, we must have either $\mathbf{B}\equiv0$ and $M$ is totally geodesic or $\abs{\mathbf{B}}^2=\frac{10}{7}\left(1+\Theta^2\right)$ and $\lambda_1, \lambda_2$ are constants and $\mu_1=\mu_2=0$, which must be the minimal Calabi torus by \cite{LS18} since $M$ is closed.
\end{proof}

\section{Proof of  \autoref{thm1}}
In this section, we will give a proof of \autoref{thm1}. Firstly we need to prove several lemmas about the function $\Theta$.

Let $SM$ be the unit tangent bundle of $M$ and $S_pM$ the fibre of the unit tangent bundle of $M$ at $p\in M$. We have the following characterization of the function $\Theta$.
\begin{lem}
\begin{align}\label{Theta}
    \Theta(p)=\max_{X\in S_pM}\sigma\left(X,X,X\right).
\end{align}
\end{lem}
\begin{proof} It is a straitforward verification. For readers' convenience, we list a proof here.

It suffices to prove that
\begin{align*}
   \max_{X\in S_pM}\sigma\left(X,X,X\right)\geq \Theta(p).
\end{align*}
Assume for some $u\in S_pM$,
\begin{align*}
    \Theta(p)=\max_{X\in S_pM}\abs{\mathbf{B}\left(X,X\right)}=\abs{\mathbf{B}\left(u,u\right)}.
\end{align*}
Choose a local orthonormal basis $\set{e_i}$ of $T_pM$ such that $e_1=u$. Applying the maximum principle, 
\begin{align*}
    \hin{\mathbf{A}^{\mathbf{B}\left(u,u\right)}\left(e_1\right)}{e_j}=\hin{\mathbf{B}\left(e_1,e_1\right)}{\mathbf{B}\left(e_1,e_j\right)}=0,\quad \forall j>1.
\end{align*}
Here $\mathbf{A}^{\nu}$ is the shape operator associated with the normal vector $\nu$. Thus,
\begin{align*}
    \mathbf{A}^{Ju}\mathbf{A}^{Ju}\left(e_1\right)= \mathbf{A}^{Je_1}\left(-J\mathbf{B}\left(e_1,e_1\right)\right)=\mathbf{A}^{\mathbf{B}\left(u,u\right)}\left(e_1\right)=\abs{\mathbf{B}\left(e_1,e_1\right)}^2e_1.
\end{align*}
We conclude that
\begin{align*}
    \mathbf{A}^{Ju}\left(e_1\right)=\pm\abs{\mathbf{B}\left(e_1,e_1\right)}e_1.
\end{align*}
 Consequently,
\begin{align*}
    \Theta(p)&=\max_{X\in S_pM}\abs{\mathbf{B}\left(X,X\right)}
    =\abs{\mathbf{B}\left(e_1,e_1\right)}=\pm\hin{\mathbf{A}^{Ju}\left(e_1\right)}{e_1}=\pm\sigma\left(e_1,e_1,e_1\right)\leq\max_{X\in S_pM}\sigma\left(X,X,X\right).
\end{align*}
\end{proof}
In the rest of this paper we  will use the equivalent description  \eqref{Theta} of $\Theta$.
\begin{lem}\label{lem:Theta}
$\Theta$ is a nonnegative Lipschitz function on $M$.
\end{lem}
\begin{proof}
It suffices to prove that for every $p_1, p_2\in M$,
\begin{align*}
    \Theta(p_1)\leq\Theta(p_2)+\max_{x\in M}\abs{\nabla\sigma}\mathrm{dist}(p_1,p_2).
\end{align*}
Choose a geodesic $\gamma:[0,\rho]\To M$ connecting $p_1$ and $p_2$, i.e., $\gamma(0)=p_1,\gamma(1)=p_2$, where $\rho=\mathrm{dist}(p_1,p_2)$. Assume
\begin{align*}
    \Theta(p_1)=\sigma(e,e,e),
\end{align*}
where $e\in S_{p_1}M$. We can extend $e$ to a tangent unit vector field $e(t)\in T_{\gamma(t)}M$ along $\gamma(t)$ by parallel transport. Consider a function $f(t)=\sigma\left(e(t),e(t),e(t)\right)$, then
\begin{align*}
    \Theta(p_1)-\Theta(p_2)\leq f(0)-f(\rho)=f'(t_0)\rho=\left(\nabla_{\dot\gamma\left(t_0\right)}\sigma\right)\left(e\left(t_0\right),e\left(t_0\right),e\left(t_0\right)\right)\rho,
\end{align*}
where $t_0\in(0,\rho)$. Therefore,
\begin{align*}
    \Theta(p_1)-\Theta(p_2)\leq \max_{ M}\abs{\nabla\sigma}\mathrm{dist}(p_1,p_2).
\end{align*}
\end{proof}

Choose $e_1(p)\in S_pM$ such that
\begin{align*}
    \Theta(p)=\sigma\left(e_1(p),e_1(p),e_1(p)\right).
\end{align*}
Apply the method of Lagrange multipliers to obtain
\begin{align*}
    \sigma\left(e_1(p),e_1(p),X\right)=\Theta(p)\hin{e_1(p)}{X},\quad\forall X\in T_pM.
\end{align*}
Hence we can choose an orthonormal basis $\set{e_1(p), e_2(p),\dotsc, e_n(p)}$ of $T_pM$ such that
\begin{align*}
    \sigma\left(e_1(p),e_i(p),e_j(p)\right)=\mu_i(p)\delta_{ij},
\end{align*}
where $\Theta(p)=\mu_1(p)\geq\mu_2(p)\geq\dotsm\geq\mu_n(p)$. Applying the maximum principle, one can check by definition directly that $\mu_1(p)\geq2\mu_2(p)$. We say that $\Theta(p)$ is of multiplicity one if \begin{align*}
    e\in S_pM,\ \sigma\left(e,e,e\right)=\sigma\left(e_1(p),e_1(p),e_1(p)\right)=\Theta(p)\quad\Longrightarrow\quad e=e_1(p).
\end{align*}

\begin{lem}\label{Theta2}
If $\mu_1(p_0)>2\mu_2(p_0)$ and  $\Theta(p_0)$ is of multiplicity one, then there is a unique smooth unit tangent vector field $e$ around a neighborhood $U\subset M$ of $p_0$ with $e(p_0)=e_1(p_0)$ such that
\begin{align*}
    \Theta(p)=\sigma\left(e(p),e(p),e(p)\right),\quad\forall p\in U.
\end{align*}
\end{lem}

\begin{proof}

Consider a smooth map
\begin{align*}
    f:SM\times\mathbb{R}\To TM,\quad \left((p,u),\lambda\right)\mapsto\left(p, \lambda u-\sum_{j=1}^n\sigma\left(u,u,e_j(p)\right)e_j(p)\right)\eqqcolon\left(p,h(p,u,\lambda)\right),
\end{align*}
where $\set{e_j(p)}_{j=1}^n$ is an orthonormal basis of $T_pM$.
We compute
\begin{align*}
 \dif h\left(\theta_k\right)=&\lambda\theta_k-2\sigma\left(u,\theta_k,e_j(p)\right)e_j(p),\quad k=1,\dotsc,n-1,\\
    \dif h\left(\dfrac{\partial}{\partial\lambda}\right)=&u,
\end{align*}
where $\set{\theta_k}_{k=1}^{n-1}$ is an  orthonormal basis of $T_{u}\left(S_{p}M\right)$. Notice that $\set{e_2(p_0),\dotsc, e_n(p_0)}$ is an orthonormal basis of $T_{e_1(p_0)}\left(S_{p_0}M\right)$. The assumption gives us
\begin{align*}
    \det\left(\left.\dif f\right\vert_{\left((p_0,e_1(p_0)),\mu_1(p_0)\right)}\right)=\Pi_{j=2}^n\left(\mu_1(p_0)-2\mu_j(p_0)\right)\neq0.
\end{align*}
Apply the inverse function theory to conclude that  $f:\Omega\To f(\Omega)$ is a diffeomorphism for some neighborhood $\Omega\subset SM\times\mathbb{R}$  of $\left(p_0,e_1(p_0),\mu_1(p_0)\right)$. In particular, for some neighborhood $\hat U\subset M$ of $p_0$, there is $\lambda\in C^{\infty}\left(\hat U\right)$ and $e\in\Gamma\left(S\hat U\right)$ such that $\lambda(p_0)=\Theta(p_0), u(p_0)=e_1(p_0)$ and
\begin{align*}
    \sigma\left(u(p),u(p),e_j(p)\right)e_j(p)=\lambda(p) u(p),\quad p\in \hat U.
\end{align*}

Consider a second order tensor $\phi(\cdot,\cdot)=\sigma\left(u,\cdot,\cdot\right)$ on $\hat U$. Let $\lambda_1\geq\lambda_2\geq\dotsm\geq\lambda_n$ be the eigenvalues of $\phi$. One can check that $\lambda_1=\lambda$ with eigenvector $u$ and $\lambda_k(p_0)=\mu_k(p_0)$. Moreover, each $\lambda_k$ is local Lipschitz in $\hat U$.  Assume $\Theta(p)=\sigma\left(e,e,e\right)$ where $e=e(p)=\cos t u(p)+\sin tv\in T_pM$ and $v\perp u(p)$. By assumption, choose $\varepsilon>0$ such that 
\begin{align*}
    \Theta(p_0)=\lambda_1(p_0)>\left(2+\varepsilon\right)\lambda_2(p_0).
\end{align*}
If $\cos t<1$, we claim that
\begin{align*}
    \cos t\leq\dfrac{1}{1+\varepsilon},
\end{align*}
in a neighborhood $\tilde U\subset \hat U$ of $p_0$. In fact, by the definition of $\Theta$, we have
\begin{align*}
    \sigma\left(e,e,u(p)\right)=\Theta(p)\hin{e}{u(p)},
\end{align*}
which implies that
\begin{align*}
    \cos t\Theta(p)=\cos^2t\lambda(p)+\sin^2t\phi(v,v).
\end{align*}
Without loss of generality, assume $0<\cos t<1$, then
\begin{align*}
    \Theta(p)=\cos t\lambda(p)+\dfrac{\sin^2t}{\cos t}\phi\left(v,v\right).
\end{align*}
We get
\begin{align*}
    \Theta(p)\leq\dfrac{1+\cos t}{\cos t}\phi\left(v,v\right)\leq\dfrac{1+\cos t}{\cos t}\lambda_2(p).
\end{align*}
Thus $\lambda_2(p)>0$ and
\begin{align*}
    \cos t\leq\dfrac{\lambda_2(p)}{\Theta(p)-\lambda_2(p)}.
\end{align*}
By the continuity of $\Theta$ and $\lambda_2$, we prove the claim.

Now we claim that there is a neighborhood $U\subset\tilde U$ of $p_0$ such that $e(p)=u(p)$ for all $p\in U$. Otherwise, there are sequences $\set{p_n}\subset U, \set{v_n\in S_{p_n}M}, \set{t_n}\subset[0,2\pi]$ such that $e_n=\cos t_n u(p_n)+\sin t_n v_n$ satisfying
\begin{align*}
    \Theta(p_n)=\sigma\left(e_n,e_n,e_n\right),\quad\cos t_n\leq\dfrac{1}{1+\varepsilon},\quad\forall n.
\end{align*}
Without loss of generality, assume $\lim_{n\to\infty}q_n=p, \lim_{n\to\infty}t_n=t,\quad\lim_{n\to\infty}v_n=v$. Since $\Theta$ is continuous according to  \autoref{lem:Theta},  $u\coloneqq\lim_{n\to\infty}u_n=\cos t u(p_0)+\sin t v\neq e_1(p_0)$ satisfies
\begin{align*}
    \Theta(p_0)=\sigma\left(u(p_0),u(p_0),u(p_0)\right),
\end{align*}
which is a contradiction.
\end{proof}
Now we are prepared to give a proof of our main \autoref{thm1}.
\begin{proof}[Proof of \autoref{thm1}]
When $n=2$ we have $\abs{\mathbf{B}}^2=4\Theta^2$, therefore \eqref{con} is equivalent to $\abs{\mathbf{B}}^2\leq2$ and  we get the conclusion from \cite{YKM}.

We need to check the case $n\geq3$. Assume that $M$ is not totally geodesic and $\Theta$ achieves its maximum value at $p_0$. Choose an orthonormal basis $\set{e_i}_{i=1}^n$ of $T_{p_0}M$ such that
\begin{align*}
\mu_1=\sigma_{111}\left(p_0\right)=\Theta(p_0).
\end{align*}
One can check that
\begin{align*}
    \sigma_{11j}\left(p_0\right)=0,\quad j=2,\dotsc,n.
\end{align*}
Thus, we may assume that 
\begin{align*}
\sigma_{1jk}(p_0)=\mu_j\delta_{jk},\quad 1\leq j, k\leq n.
\end{align*}

For each $e_i$, choose a geodesic $\gamma:\left(-\varepsilon,\varepsilon\right)\To M$ with $\gamma(0)=p_0, \dot\gamma(0)=e_i$. We move  $e_1\in T_{p_0}M$ along the geodesic $\gamma(t)$ to $e_1(t)\in T_{\gamma(t)}M$ by parallel transport. Consider a function $f:\left(-\varepsilon,\varepsilon\right)\To\mathbb{R}$ defined by
\begin{align*}
    f(t)=\sigma\left(e_1(t),e_1(t),e_1(t)\right).
\end{align*}
Then $f(t)$ achieves its local maximum value at $t=0$. The maximum principle gives 
\begin{align*}
    0\geq f''(0)=\sigma_{111,ii}(p_0).
\end{align*}
Thus
\begin{align*}
\Delta\sigma_{111}(p_0)\leq0.
\end{align*}
Now applying the Simons' identity \eqref{eq:simons}, we have at $p_0$
\begin{align*}
0\geq&\left(n+1\right)\sigma_{111}+2\sigma_{1ab}\sigma_{1bc}\sigma_{1ca}-3\sigma_{1ab}\sigma_{abc}\sigma_{11c}\\
=&\left(n+1\right)\mu_1+2\sum_{j}\mu_{j}^3-3\mu_1\sum_{j}\mu_j^2\\
=&\left(n+1\right)\mu_1-\mu_1^3+2\sum_{j>1}\mu_{j}^3-3\mu_1\sum_{j>1}\mu_j^2.
\end{align*}

Without loss of generality, assume $\mu_1>0$.
Set
\begin{align*}
a_j=-\mu_{j},\quad j>1.
\end{align*}
Since $\sum_{j}\mu_j=0$, we get
\begin{align*}
0\geq&\left(n+1\right)\sum_{j>1}a_j-\left(\sum_{j>1}a_j\right)^3-2\sum_{j>1}a_{j}^3-3\sum_{j}a_j\sum_{k>1}a_k^2\\
=&\left(n+1\right)\sum_{j>1}a_j-6\left(\sum_{j>1}a_j\right)^3+12\sum_{i>1}a_i\sum_{j>k>1}a_{j}a_k-6\sum_{i>j>k>1}a_ia_ja_k.
\end{align*}
By Newton's inequality,
\begin{align*}
\sum_{i>j>k>1}a_ia_ja_k\leq\dfrac{2(n-3)}{3(n-2)}\dfrac{\left(\sum_{j>k>1}a_{i}a_j\right)^2}{\sum_{i>1}a_i},
\end{align*}
the equality holds if and only if
\begin{align*}
   \mu_2=\mu_3=\dotsm=\mu_n.
\end{align*}

We obtain
\begin{align}\label{eq:4.2}
0\geq&n+1-6\left(\sum_{j>1}a_j\right)^2+12\sum_{j>k>1}a_{j}a_k-\dfrac{4(n-3)}{n-2}\dfrac{\left(\sum_{j>k>1}a_{j}a_k\right)^2}{\left(\sum_{i>1}a_i\right)^2}.
\end{align}
Denoted by
\begin{align*}
\beta=\mu_1^2+3\sum_{j>1}\mu_j^2=4\left(\sum_{j>1}a_j\right)^2-6\sum_{j>k>1}a_ja_k,
\end{align*}
then (cf. \cite[Page 11]{LS18})
\begin{align*}
    \abs{\mathbf{B}}^2\geq\beta\geq\dfrac{n+2}{n-1}\mu_1^2.
\end{align*}
Then \eqref{eq:4.2} gives
\begin{align*}
0\geq&n+1+2\mu_1^2-2\beta-\dfrac{n-3}{9(n-2)}\left(4\mu_1-\dfrac{\beta}{\mu_1}\right)^2.
\end{align*}

If $n=3$, then
\begin{align*}
\beta\geq2+\mu_1^2.
\end{align*}
In general,
\begin{align*}
0\geq&n+1+2\mu_1^2-2\beta-\dfrac{n-3}{9(n-2)}\left(4\mu_1-\dfrac{\beta}{\mu_1}\right)^2\\
=&n+1+2\mu_1^2-2\beta-\dfrac{n-3}{9(n-2)}\left(16\mu_1^2-8\beta+\dfrac{\beta^2}{\mu_1^2}\right)\\
=&n+1+\dfrac{2(n+6)}{9(n-2)}\mu_1^2-\dfrac{2(5n-6)}{9(n-2)}\beta-\dfrac{n-3}{9(n-2)}\dfrac{\beta^2}{\mu_1^2}\\
=&n+1+\dfrac{2(n+6)}{9(n-2)}\left(\dfrac{5n-6}{2(n+6)}\dfrac{\beta}{\mu_1}-\mu_1\right)^2-\dfrac{3n}{2(n+6)}\dfrac{\beta^2}{\mu_1^2}.
\end{align*}
Notice that
\begin{align*}
\beta\geq\dfrac{n+2}{n-1}\mu_1^2,
\end{align*}
we know that
\begin{align*}
\dfrac{5n-6}{2(n+6)}\dfrac{\beta}{\mu_1}-\mu_1\geq\dfrac{n-1}{n+2}\dfrac{\beta}{\mu_1}-\mu_1\geq0.
\end{align*}
Therefore,
\begin{align*}
\dfrac{5n-6}{2(n+6)}\dfrac{\beta}{\mu_1}-\sqrt{\dfrac{27n(n-2)}{4(n+6)^2}\dfrac{\beta^2}{\mu_1^2}-\dfrac{9(n+1)(n-2)}{2(n+6)}}\leq\mu_1\leq\dfrac{n-1}{n+2}\dfrac{\beta}{\mu_1},
\end{align*}
which implies
\begin{align*}
\beta\geq\dfrac{n+2}{\sqrt{n}}\mu_1.
\end{align*}
We conclude that
\begin{align*}
\abs{\mathbf{B}}^2\geq\dfrac{n+2}{\sqrt{n}}\mu_1.
\end{align*}

Therefore under the assumption (\ref{con}), we must have $\abs{\mathbf{B}}^2=\frac{n+2}{\sqrt{n}}\mu_1$ at $p_0$ and
\begin{align}
\sigma_{1jk}\left(p_0\right)=\mu_{j}\delta_{jk},\quad j, k=1,\dotsc,n \label{eq:calabi_torus1}\\
\sigma_{ijk}\left(p_0\right)=0,\quad i, j, k>1. \label{eq:calabi_torus2}
\end{align}
Moreover $\mu_2=\dotsm=\mu_n=-\frac{1}{n-1}\mu_1<0$.

\begin{claim}
$\Theta$ is a constant on $M$.
\end{claim}
If this Claim is true, then the previous argument claims that the condition \eqref{eq:calabi_torus1} and \eqref{eq:calabi_torus2} hold everywhere. As an immediately consequence, $M$ is the Calabi torus (cf. \cite{LS18}) since $M$ is closed.

Now we prove the the above Claim as follows.
\begin{proof}[Proof of the Claim] Consider a nonempty subset $\Omega$ of $M$ defined by
\begin{align*}
    \Omega\coloneqq\set{p\in M: \Theta(p)=\mu_1}.
\end{align*}
Since $\Theta$ is continuous, we know that $\Omega$ is a closed subset of $M$. Then it suffices to prove that $\Omega$ is also an open subset of $M$ since $M$ is connected.

Firstly we claim that $\Theta(p_0)$ is of multiplicity one, i.e., the unit tangent vector $e\in T_{p_0}M$ with $\sigma\left(e,e,e\right)=\mu_1$ is unique. In fact, put $e=\sum_{i}x^ie_i$, then $\sum_jx^jx^j=1$ and
\begin{align*}
\mu_1=&\sigma\left(e,e,e\right)\\
=&x^ix^jx^k\sigma_{ijk}\\
=&x^1x^1x^1\mu_1+3x^1\sum_{j>1}x^jx^j\mu_j\\
=&x^1x^1x^1\mu_1-\dfrac{\mu_1}{n-1}x^1\left(1-x^1x^1\right).
\end{align*}
We must have $x^1=1$ and $e=e_1$. Therefore, by \autoref{Theta2} we can extend $e_1$ to a smooth tangent vector field still denoted by $e_1$ around a neighborhood $U$ of $p_0$, such that for all $p\in U$ we have
\begin{align*}
\Theta(p)=\sigma(e_1,e_1,e_1)(p).
\end{align*}

Secondly, we claim that $\Theta$ is subharmonic in $U$. In fact,  since
\begin{align*}
    \nabla_{e_j}\Theta=\sigma_{111,j}+3\sigma_{11i}\hin{\nabla_{e_j}e_1}{e_i}=\sigma_{111,j}=\sigma_{11j,1},
\end{align*}
and for $j>1$,
\begin{align*}
    0=&\nabla_{e_1}\sigma_{11j}\\
    =&\sigma_{11j,1}+2\sigma_{1kj}\hin{\nabla_{e_1}e_1}{e_k}+\sigma_{11k}\hin{\nabla_{e_1}e_j}{e_k}\\
    =&\sigma_{11j,1}+2\sum_{k>1}\sigma_{1kj}\hin{\nabla_{e_1}e_1}{e_k}+\sigma_{111}\hin{\nabla_{e_1}e_j}{e_1}\\
    =&\sigma_{11j,1}+2\sum_{k>1}\sigma_{1kj}\hin{\nabla_{e_1}e_1}{e_k}-\sigma_{111}\hin{\nabla_{e_1}e_1}{e_j}.
    \end{align*}
At a considered point $p\in U$, we may assume $\sigma_{1jk}=\sigma_{1jj}\delta_{jk}$ and $\nabla_{e_i}e_j=0$ for $i,j>1$. We get
\begin{align*}
    \nabla_{e_1}\Theta=&\sigma_{111,1},\\
    \nabla_{e_j}\Theta=&\left(\Theta-2\sigma_{1jj}\right)\hin{\nabla_{e_1}e_1}{e_j},\quad j>1,
\end{align*}
and
\begin{align*}
    \Delta\Theta=&\nabla_{e_j}\nabla_{e_j}\Theta-\nabla_{\nabla_{e_j}e_j}\Theta\\
    =&\left(\sigma_{111,11}+4\sum_{j>1}\sigma_{111,j}\hin{\nabla_{e_1}e_1}{e_j}\right)\\
    &+\sum_{j>1}\left(\sigma_{111,jj}
    +3\sum_{k}\sigma_{11k,j}\hin{\nabla_{e_j}e_1}{e_k}+\sum_{k}\sigma_{111,k}\hin{\nabla_{e_j}e_j}{e_k}\right)\\
    &-\sum_{j>1}\nabla_{e_j}\Theta\hin{\nabla_{e_1}e_1}{e_j}\\
    =&\Delta\sigma_{111}+3\sum_{j>1}\nabla_{e_j}\Theta\hin{\nabla_{e_1}e_1}{e_j}\\
    =&\Delta\sigma_{111}+3\sum_{j>1}\left(\Theta-2\sigma_{1jj}\right)\hin{\nabla_{e_1}e_1}{e_j}^2,
\end{align*}
where in the second equality we used the symmetry of the fourth order tensor $(\sigma_{ijk,l})$. By the assumption \eqref{con}, the previous argument implies that
\begin{align*}
    \Delta\sigma_{111}\geq0.
\end{align*}
Therefore we have  for all $p\in U$,
\begin{align*}
\Delta \Theta\geq3\sum_{i>1}\left(\Theta-2\sigma_{1ii}\right)\hin{\nabla_{e_1}e_1}{e_i}^2.
\end{align*}
One can check that $\sigma_{11i}(p)=0$ and $\Theta(p)\geq2\sigma_{1ii}(p)$ for all $i>1$ and all $p\in U$. Thus $\Theta$ is subharmonic in $U$. Since $\Theta$ achieved its local maximum value at $p_0\in U$, the strong maximum principle implies that $\Theta$ locally must be a constant in $U$. We conclude that $p_0$ is an interior point of $\Omega$. Thus $\Omega$ is an open subset of $M$. Therefore, $\Theta$ is a constant on $M$. 
\end{proof}

At the end let us show that when $n=3$ condition $(\ref{con})$ implies $\abs{\mathbf{B}}^2\leq2+\Theta^2$. Since $\frac{5}{2}\Theta^2\leq\abs{\mathbf{B}}^2\leq\frac{5}{\sqrt{3}}\Theta$, we have $\Theta\leq\frac{2}{\sqrt{3}}$, and $\Theta^2+2-\frac{5}{\sqrt{3}}\Theta=\left(\Theta-\frac{2}{\sqrt{3}}\right)^2+\frac{2}{3}-\frac{1}{\sqrt{3}}\Theta\geq0$. Therefore $\abs{\mathbf{B}}^2\leq\frac{5}{\sqrt{3}}\Theta$ implies $\abs{\mathbf{B}}^2\leq2+\Theta^2$. This completes the proof of  \autoref{thm1}.

\end{proof}

\appendix
\section{An application to Lagrangian submanifolds in the nearly K\"ahler \texorpdfstring{$\mathbb{S}^6$}{ unit six sphere}}
Here we  give a slight improvement of the main theorem in \cite{HYY} as follows, by the argument used in the proof of  \autoref{thm2}.

\begin{theorem}\label{thm3}
Let $M$ be a closed Lagrangian submanifold in the homogeneous nearly K\"ahler $\mathbb{S}^6$. Then  we have
\begin{align}\label{Integ1}
\int_{M}\abs{\mathbf{B}}^2\left(\abs{\mathbf{B}}^2-\frac{75}{56}-\frac{10}{7}\Theta^2\right)\geq0.
\end{align}
Moreover, the equality in \eqref{Integ1} holds if and only if $M$ is either the totally geodesic sphere, or the
Dillen-Verstraelen-Vrancken's Berger sphere (see  \cite[Theorem 5.1]{DVV}) which satisfies $\abs{\mathbf{B}}^2=\frac{75}{56}+\frac{10}{7}\Theta^2$ with $\abs{\mathbf B}^2\equiv\frac{25}{8}$ and $\Theta\equiv\frac{\sqrt{5}}{2}$.
\end{theorem}

\begin{proof}
Here we only give a brief sketch. For more details please see \cite{HYY}. We identify $\mathbb{R}^7$ as the imaginary Cayley numbers.  The Cayley multiplication induces a cross product $``\times"$ on $\mathbb{R}^7$. The almost complex structure $J$ on $\mathbb{S}^6\subset\mathbb{R}^7$ is then given by
\begin{align*}
    JX\coloneqq x\times X,\quad\forall X\in T\mathbb{S}^6.
\end{align*}
Let $\bar\nabla$ be the Levi-Civita connection on $\mathbb{S}^6$, then $\left(\bar\nabla_XJ\right)X=0$ for all $X\in T\mathbb{S}^6$. Then $\omega_{ijk}=\hin{\left(\bar\nabla_{e_i}J\right)e_j}{Je_k}$ is the volume form of $M$. Since $M$ is Lagrangian, i.e., $JTM\subset T^{\bot}M$,  we have (\cite[Lemma 3.2]{SS})
\begin{align*}
\mathbf{B}\left(e_i,\left(\bar\nabla_{e_j}J\right)e_k\right)=J\left(\bar\nabla_{\mathbf{B}\left(e_i,e_j\right)}J\right)e_k+J\left(\bar\nabla_{e_j}J\right)\mathbf{B}\left(e_i,e_k\right),
\end{align*}
which implies that $M$ is minimal (cf. \cite[Theorem 1]{Ejiri}). We have the following Simons' identity (cf. \cite{CCK})
\begin{align*}
    \dfrac12\Delta\abs{\mathbf{B}}^2=\abs{\nabla^{\bot}\mathbf{B}}^2+3\abs{\mathbf{B}}^2-\sum_{\alpha,\beta=1}^3\hin{\mathbf{A}^{\nu_{\alpha}}}{\mathbf{A}^{\nu_{\beta}}}^2-\sum_{\alpha,\beta=1}^3\abs{\left[\mathbf{A}^{\nu_{\alpha}},\mathbf{A}^{\nu_{\beta}}\right]}^2.
\end{align*}
Here  $\set{\nu_{\alpha}}$ is a local orthonormal frames of $T^{\bot}M$.
Set
\begin{align*}
    \sigma_{ijk}=\hin{\mathbf{B}\left(e_i,e_j\right)}{Je_k},
\end{align*}
then $\sigma$ is a third order symmetric tensor. One can check that
\begin{align*}
    \sigma_{ijk,l}=\hin{\left(\nabla_{e_i}^{\bot}\mathbf{B}\right)\left(e_j,e_k\right)}{Je_l}.
\end{align*}
Introduce 
\begin{align*}
    u_{ijkl}\coloneqq&\dfrac14\left(\sigma_{ijk,l}+\sigma_{jkl,i}+\sigma_{kli,j}+\sigma_{lij,k}\right)\\
    =&\sigma_{ijk,l}-\dfrac14\left(\sigma_{jkm}\omega_{lim}+\sigma_{ikm}\omega_{ljm}+\sigma_{ijm}\omega_{lkm}\right).
\end{align*}
 One can check that $u$ is a fourth order symmetric tensor and $\sum_{i}u_{iij,k}=0$. By using the fact $\sigma_{ijk,l}=\sigma_{ijl,k}+\sigma_{ijm}\omega_{lkm}$, a direct calculation yields (cf. \cite[emma 4.4]{HYY})
\begin{align*}
    \abs{\nabla^{\bot}\mathbf{B}}^2=\abs{u}^2+\dfrac34\abs{\mathbf{B}}^2.
\end{align*}
We therefore obtain
\begin{align*}
\dfrac12\Delta\abs{\sigma}^2=&\abs{u}^2+\dfrac{15}{4}\abs{\sigma}^2-\sum_{i,j=1}^3\hin{\sigma_i}{\sigma_j}^2-\sum_{i,j=1}^3\abs{[\sigma_i,\sigma_j]}^2,
\end{align*}
where $\sigma_i=\left(\sigma_{ijk}\right)_{1\leq j,k\leq n}$. Then, similarly with the proof of \autoref{thm2}, we obtain
\begin{equation*}
\frac12\Delta\abs{\mathbf{B}}^2\geq
\frac{14}5\left(\frac{75}{56}+\dfrac{10}{7}\Theta^2-\abs{\mathbf{B}}^2\right)\abs{\mathbf{B}}^2.
\end{equation*}
The rest of the proof follows from that of \cite{HYY}.
\end{proof}

\quad

\biboptions{longnamesfirst,sort&compress}

\end{document}